\documentclass[11pt]{amsart}

\usepackage{mathrsfs}
\usepackage{amsmath, amsthm, amsfonts}
\usepackage[percent]{overpic}

\usepackage{color}
\usepackage{MnSymbol}
\usepackage{fullpage}
\usepackage{bbm}

\usepackage{thmtools}
\usepackage{thm-restate}

\usepackage{hyperref}

\usepackage{cleveref}

\declaretheorem[name=Theorem,numberwithin=section]{thm}

\newtheorem*{thm*}{Theorem}
\newtheorem{cor}[thm]{Corollary}
\newtheorem{prop}[thm]{Proposition}
\newtheorem{lem}[thm]{Lemma}

\theoremstyle{definition}
\newtheorem{defn}[thm]{Definition}

\newtheorem*{ack}{Acknowledgements}
\newtheorem{ex}[thm]{Example}
\theoremstyle{remark}

\newcommand{\QQ}{\mathbb{Q}}

\newcommand{\RR}{\mathbb{R}}

\newcommand{\OO}{\mathcal{O}}

\newcommand{\PP}{\mathbb{P}}

\newcommand{\Pic}{\text{Pic}}

\newcommand{\Mgb}{\overline{\mathcal{M}}_g}
\newcommand{\Mgnb}{\overline{\mathcal{M}}_{g,n}}
\newcommand{\MOnb}{\overline{\mathcal{M}}_{0,n}}

\newcommand{\rk}{\text{rk}}

\newcommand{\M}[2]{\mathcal{M}_{{#1}, {#2}}}
\newcommand{\Mbar}[2]{\overline{\mathcal{M}}_{{#1}, {#2}}}

\newtheorem{Thm*}{Theorem*}
\theoremstyle{definition}

\title{Non-polyhedral effective cones from the moduli space of curves 
}
\author{Scott Mullane}
\date{\today}

\AtEndDocument{\bigskip{\footnotesize%
\textsc{Scott Mullane, Institut f\"{u}r Mathematik, Goethe-Universit\"{a}t Frankfurt, Robert-Mayer-Str. 6-8, 60325 Frankfurt am Main, Germany}\par
\textit{E-mail address}: \texttt{mullane@math.uni-frankfurt.de} \par
}}

\begin{document}
\thispagestyle{empty}

\maketitle

\begin{abstract}
We show that the pseudoeffective cone of divisors $\overline{\text{Eff}}^1(\Mbar{g}{n})$ for $g\geq 2$ and $n\geq 2$ is not polyhedral by showing that the class of the fibre of the morphism forgetting one point forms an extremal ray of the dual nef cone of curves $\overline{\text{Nef}}_1(\Mbar{g}{n})$ and the cone at this ray is not polyhedral.
\end{abstract}

\setcounter{tocdepth}{1}

\tableofcontents

\section{Introduction}
The birational geometry of a projective variety is in many ways dictated by the structure of the cone of effective divisors.
For example, a projective variety with canonical singularities is of general type if the canonical divisor lies in the interior of the effective cone. This question has attracted much attention in the case of the moduli space of curves~\cite{HarrisMumford},\cite{Harris},\cite{EisenbudHarrisKodaira},\cite{Farkas23},\cite{FarkasKoszul}. Further, a $\QQ$-factorial projective variety with finitely generated Cox ring is known as a \emph{Mori dream space} because the effective cone is polyhedral and decomposes into finitely many convex chambers each providing a birational model of the variety. In the case of moduli spaces, these models often carry modular significance. The question of when the moduli space of curves is a Mori dream space has also attracted much interest~\cite{CastravetTevelev},\cite{GK},\cite{HKL},\cite{Keel2},\cite{HuKeel}. In particular, $\Mbar{g}{n}$ cannot be a Mori dream space if the effective cone is non-polyhedral~\cite{ChenCoskun},\cite{Mullane1}, though understanding the structure of the effective cone of $\Mgnb$ is strongly motivated beyond this well-progressed question.

\subsection{History}
Many authors have contributed to the current understanding of the structure of the effective cone of the moduli space of curves. The effective cone of $\Mbar{0}{5}$ is generated by boundary divisors, though Hassett and Tschinkel~\cite{HT} showed the effective cone of $\Mbar{0}{6}$ is generated by the boundary and Keel-Vermiere divisors~\cite{KV} that intersect the interior of the moduli space and were originally provided as a counter example to Fulton's conjecture that cones of effective cycles in $\Mbar{0}{n}$ in all codimension were generated by boundary classes. Castravet and Tevelev~\cite{CastravetTevelevHypertree} generalised the divisors of Keel and Vermiere through Brill-Noether theory on singular higher genus curves to produce finitely many extremal rays in the effective cone of $\Mbar{0}{n}$ for each $n\geq 7$ indexed by irreducible hypertrees and conjectured that with the boundary divisors these rays generated the effective cone. Opie~\cite{Opie} found new extremal rays that contradicted this conjecture. The Picard rank of both $\Mbar{1}{2}$ and $\overline{\mathcal{M}}_2$ is two and in both cases the effective cone is generated by the two irreducible components of the boundary. However, for $\Mbar{1}{n}$ with $n\geq3$, Chen and Coskun~\cite{ChenCoskun} exhibited infinitely many extremal effective divisors via relations on the marked points in the group law of the elliptic curve on which they lie and hence showed these effective cones are not rational polyhedral. 
Rulla~\cite{Rulla} fully described the polygonal effective cones of $\Mbar{2}{1}$, $\overline{\mathcal{M}}_3$ and exhibited $7$ extremal rays in the Picard rank $6$ effective cone of $\Mbar{2}{2}$, showing this cone is not simplicial, however, recently effective divisors not accounted for by these extremal rays have been identified~\cite{Mullane2}. In general genus $g\geq 2$, Farkas and Verra~\cite{FarkasVerraU},\cite{FarkasVerraTheta} gave for each fixed $g$ and $n$ with $g-2\leq n\leq g$, a finite number of extremal divisors in $\Mbar{g}{n}$. For $\Mbar{g}{n}$ with $g\geq2$ and $n\geq g+1$, the author~\cite{Mullane1} exhibited infinitely many extremal divisors coming from the strata of meromorphic differentials with fixed multiplicities of zeros and poles and hence showed in these cases that the effective cone is not rational polyhedral.

\subsection{Main results}
Let $[F]$ be the curve class obtained as the general fibre of the morphism $\pi:\Mbar{g}{n}\longrightarrow \Mbar{g}{n-1}$ that forgets the $n$th marked point. Irreducible curves with class equal to $[F]$ cover an open dense subset of $\Mbar{g}{n}$. Any effective divisor with negative intersection with $[F]$ would need to contain this dense subset implying $[F]$ must have non-negative intersection with every pseudoeffective divisor and is hence nef.

For any nef curve class $[B]$ we define the pseudoeffective dual space
$$[B]^\vee:=\{[D]\in \overline{\text{Eff}}^1(\Mbar{g}{n})\hspace{0.3cm}\big|\hspace{0.3cm} [B]\cdot[D]=0\}.$$
and obtain the following two theorems.
\begin{restatable}{thm}{Dual}
\label{Thm:dual} 
For $g\geq 2$ and $n\geq2$, 
$$\rho(\Mbar{g}{n})-n\leq\text{rank}([F]^{\vee}\otimes\RR)\leq \rho(\Mbar{g}{n})-2,$$
where $\rho(\Mbar{g}{n})$ denotes the Picard rank of $\Mbar{g}{n}$.
\end{restatable}
\begin{restatable}{thm}{Extremal}
\label{Thm:Extremal}
The curve class $[F]$ is extremal in $\overline{\text{Nef}}_1(\Mbar{g}{n})$ for $g\geq2$ and $n\geq1$.
\end{restatable}
The cone $\overline{\text{Nef}}_1(\Mbar{g}{n})$ is non-polyhedral at any extremal ray where the pseudoeffective dual space of divisors has corank $2$ or more. The main result of this paper follows from these two theorems and the duality of $\overline{\text{Nef}}_1(\Mbar{g}{n})$ and $\overline{\mbox{Eff}}^1(\Mbar{g}{n})$.
\begin{restatable}{cor}{NonPoly}
The cones $\overline{\text{Nef}}_1(\Mbar{g}{n})$ and $\overline{\mbox{Eff}}^1(\Mbar{g}{n})$ are not rational polyhedral for $g\geq2$ and $n\geq 2$.
\end{restatable}

We summarise the arguments for each of the two main theorems.

\subsection{Proof of Theorem~\ref{Thm:dual}}
Identifying extremal classes in a cone is often a difficult undertaking. In the case of effective divisors there is one well-known trick to show extremality. If irreducible curves with numerical class equal to $[B]$ cover a Zariski  dense subset of an irreducible effective divisor $E$, then $[B]$ is known as a \emph{covering curve} for $E$ and we have the following lemma.
\begin{restatable}{lem}{Covering}
\label{lemma:covering}
If $[B]$ is a covering curve for an irreducible effective divisor $E$ with $[B]\cdot [E]<0$ then $E$ is rigid and extremal.
\end{restatable}
The setting of $\Mbar{g}{2}$ offers a natural way to define a standard potential covering curve for any effective divisor by intersecting the divisor with a $2$-dimensional general fibre of the morphism $\Mbar{g}{2}\longrightarrow\overline{\mathcal{M}}_g$ that forgets the marked points. More formally, let $C$ be a general smooth genus $g$ curve and
$$i:C\times C\longrightarrow \Mbar{g}{2}$$
the natural morphism. For a fixed divisor $[D]$ in $\Mbar{g}{2}$, define
$$[B_D]:=i_*i^*[D].$$
Test surfaces in the context of the Chow ring of moduli spaces of curves were considered more generally by Faber in the last chapter of~\cite{Faber}. In Proposition~\ref{Thm:eff} we show that for any effective divisor $D$, the curve $[B_D]$ is either effective or zero. Further, if
$[D]$ is the class of an effective divisor it can be expressed as
$$[D]=c_{\psi_1}\psi_1+c_{\psi_2}\psi_2+c_\lambda\lambda+\sum_{i=0}^{g-1}c_{i:\{1,2\}}\delta_{i:\{1,2\}}+\sum_{i=1}^gc_{i:\{1\}}\delta_{i:\{1\}}$$ 
in the standard generators (see~\S\ref{Pic}). Note that some authors use a slightly different convention including negative signs before the coefficients of the boundary divisors.  

By appealing to Lemma~\ref{lemma:covering} (with some care to ensure the irreducibility of the curves with class $[B_D]$) we obtain the following.
\begin{restatable}{prop}{CoveringExtremal}
\label{Prop:Extremal}
If $E$ is an irreducible effective divisor in $\Mbar{g}{2}$ for $g\geq2$ and 
$$[B_E]\cdot [E]=(2g-2)\left((4g-4)c_{\psi_1}c_{\psi_2}+(c_{\psi_1}+c_{\psi_2})^2-c_{0:\{1,2\}}^2  \right)<0$$
then $E$ is rigid and extremal and $[B_E]$ is the class of a covering curve for $E$.
\end{restatable}
In Example~\ref{ex1} and \ref{ex2} we provide simple applications of this theorem, however, the utility of the idea is in using such a curve to identify base loci in the linear system associated to other effective divisors. If $[B_E]\cdot [D]<0$ for some other effective divisor $[D]$, then $E$ appears with some positive multiplicity in the base locus of $|kD|$ for all $k>0$. Applying this to the covering curves we have constructed in $\Mbar{g}{2}$ we obtain the following.
\begin{restatable}{prop}{BaseLociBD}
\label{Prop:BaseLociBD}
If $D$ is an effective divisor in $\Mbar{g}{2}$ for $g\geq 2$ and
$$(4g-4)c_{\psi_1}c_{\psi_2}+(c_{\psi_1}+c_{\psi_2})^2-c_{0:\{1,2\}}^2  <0$$
then there exists an irreducible rigid and extremal divisor $E$ such that $[B_{E}]\cdot[E]<0$ and $[B_{E}]\cdot [D]<0$. Further, for any $k>0$,
$$|kD|=\left|k\left(D-\frac{[B_{E}]\cdot [D]}{[B_{E}]\cdot[E]}[E]\right)\right|+   k\frac{[B_{E}]\cdot [D]}{[B_{E}]\cdot[E]}E.$$
\end{restatable}

With this proposition we proceed to prove Theorem~\ref{Thm:dual}. Consider an irreducible effective divisor $D$ in $\Mbar{g}{n}$ that intersects the interior $\M{g}{n}$ with $[F]\cdot [D]=0$. If $D$ contains the point $[C,p_1,\dots,p_{n}]\in \M{g}{n}$ then consider the curve $B$ equal to the fibre of $[C,p_1,\dots,p_{n-1}]$ under the morphism $\pi:\Mbar{g}{n}\longrightarrow \Mbar{g}{n-1}$ that forgets the $n$th point. As $[B]=[F]$ we have $[B]\cdot[D]=0$ so necessarily as $B$ is irreducible and intersects $D$ set theoretically, $B$ is contained in the support of $D$ and hence the divisor $D$ is supported on the pullback of an effective divisor under $\pi$. Further, if $D$ is an irreducible divisor in $\Mbar{g}{n}$ that does not intersect the interior $\M{g}{n}$ with $[F]\cdot [D]=0$ then $D$ is an irreducible component of the boundary other than $\delta_{0:\{i,n\}}$ for $i=1,\dots,n-1$. This shows 
$$\text{rank}([F]^{\vee}\otimes\RR)\geq \rho(\Mbar{g}{n})-n$$
and we are left to consider strictly pseudoeffective divisors.

Restricting first to the case $n=2$, consider a strictly pseudoeffective divisor class $[D]$ with $[F]\cdot[D]=0$ such that $[D]$ lies outside of $\RR\langle\psi_1-\delta_{0:\{1,2\}},\lambda,\delta_0, \{\delta_{i:\emptyset}\}_{i=1}^{g-1},\{ \delta_{i:\{1\}}  \}_{i=1}^{g-1}  \rangle$. Then the class $[D]$ must be of the form  
$$[D]=((1-2g)c_{\psi_2}-c_{0:\{1,2\}})\psi_1+c_{\psi_2}\psi_2+c_{0:\{1,2\}}\delta_{0:\{1,2\}}+c_{\lambda}\lambda+c_0\delta_0+\sum_{i=1}^{g-1}c_{i:\{1\}}\delta_{i:\{1\}}+\sum_{i=1}^{g-1}c_{i:\emptyset } \delta_{i:\emptyset }$$
with $c_{\psi_2}\ne 0$. But then
$$[B_D]\cdot [D]=-8c_{\psi_2}^2(g-1)^2g<0$$
and applying Proposition~\ref{Prop:BaseLociBD} to effective divisors limiting to $[D]$ implies the existence of a rigid extremal effective divisor $E$ with $[F]\cdot [E]=0$ and nonzero coefficient of $\psi_2$ providing a contradiction and proving Theorem~\ref{Thm:dual} for $n=2$.

Though the fibre of the forgetful morphism $\Mbar{g}{n}\longrightarrow\Mbar{g}{n-1}$ over a general point is irreducible, for $n\geq3$ the fibre over a general point in $\delta_{0:\{1,\dots,n-1\}}$ is reducible with two irreducible components. Hence $[F]$ can be written as an effective sum of these two curves which form covering curves for $\delta_{0:\{1,\dots,n-1\}}$ and $\delta_{0:\{1,\dots,n\}}$. Any extremal $[D]\in \overline{\text{Eff}}^1(\Mbar{g}{n})$ not equal to $\delta_{0:\{1,\dots,n-1\}}$ or $\delta_{0:\{1,\dots,n\}}$ must have nonnegative intersection with each of these covering curves and further $[F]\cdot[D]=0$ implies each of these intersections must in fact be zero. Hence if $\pi:\Mbar{g}{2}\longrightarrow \Mbar{g}{n}$ is the morphism gluing in a general $n$-pointed rational curve at the first point, $[F]\cdot\pi^*[D]=0$. Lemma~\ref{Lem:ps} shows in this situation $\pi^*[D]$ is pseudoeffective and hence by pulling back the class of $[D]$ we have $c_{\psi_n}=0$ providing a second linearly independent condition on $[F]^\vee\otimes\RR$ for $n\geq 3$ completing the proof of Theorem~\ref{Thm:dual}.

\subsection{Proof of Theorem~\ref{Thm:Extremal}}
To show that $[F]$ is an extremal nef curve we appeal to two interesting families of divisors from the theory of Hurwitz spaces. We define $D_k$ and $E_k$ as the following loci in $\M{g}{2}$
\begin{equation*}
D_k:={\left\{ [C,p_1,p_2]\in\M{g}{2} \big| \exists [C,p_1,p_2,q_1,\dots,q_{g-1}]\in\M{g}{g+1}, \OO_C((k(g-1)+1)p_1-p_2-k\sum_{i=1}^{g-1}q_i)\sim\OO_C   \right\}}
\end{equation*}
\begin{equation*}
E_k:={\left\{ [C,p_1,p_2]\in\M{g}{2} \big| \exists [C,p_1,p_2,q_1,\dots,q_{g-1}]\in\M{g}{g+1}, \OO_C((k(g-1)-1)p_1+p_2-k\sum_{i=1}^{g-1}q_i)\sim\OO_C   \right\}}.
\end{equation*}
Taking the closure we obtain two families of divisors $\overline{D}_k$ and $\overline{E}_k$ for $k\geq 2$ in $\Mbar{g}{2}$. In Theorem~\ref{Int} we compute the intersection of these divisors with a number of curves via Jacobian of a general curve and the admissible covers compactification of Hurwitz spaces. The table below gives two values of interest to us for each divisor class. 
\begin{center}
\begin{tabular}{|p{1.5cm}|p{3cm}|p{4cm}|  }
 \hline
  $[*]$&$[F]\cdot[*]$&Coefficient of $\psi_2$ \\
 \hline
 $[\overline{D}_k]$ &$(k^{2g - 2}-1) g$ &$\frac{1}{2}(1-k)k^{2g-2}$  \\
$[\overline{E}_k]$&$(k^{2g - 2}-1) g$  &$\frac{1}{2}(k+1)k^{2g-2}$  \\
\hline
\end{tabular}
\end{center}
These two families of divisor classes are employed because the coefficients of $\psi_2$ have opposite parity and the magnitude of the coefficient of $\psi_2$ dominates the intersection of the divisors with $[F]$ asymptotically in $k$. 

If $[F]$ is not extremal it is the sum of nef curve classes with zero intersection with $[F]^\vee$. A non-trivial nef decomposition of $[F]$ implies the existence of a nef curve class $[F^t]$ for a fixed $t\in\RR\setminus\{0\}$ with intersections
$$[F^t]\cdot\psi_1=1,\hspace{1cm}[F^t]\cdot \psi_2=(2g-1)+t,\hspace{1cm}[F^t]\cdot\delta_{0:\{1,2\}}=1$$
and zero intersection with the other standard generators of $\Pic(\Mbar{g}{2})$.  But for fixed $t>0$ 
$$ [F^t]\cdot [\overline{D}_k]= (k^{2g - 2}-1) g +t\frac{1}{2}(1-k)k^{2g-2} =\frac{-t}{2}k^{2g-1}+\OO(k^{2g-2})<0 \text{  for  }k\gg0,$$
while for fixed $t<0$ 
$$ [F^t]\cdot [\overline{E}_k]=(k^{2g - 2}-1) g   +t \frac{1}{2}k^{2g-2}(k+1)=\frac{t}{2}k^{2g-1}+\OO(k^{2g-2}) <0 \text{  for  }k\gg0.$$
Hence $[F^t]$ is nef if and only if $t=0$, which completes the proof of Theorem~\ref{Thm:Extremal} in the case that $n=2$. For $n\geq 3$ we consider the the forgetful morphisms
$$\pi_j:\Mbar{g}{n}\longrightarrow\Mbar{g}{2}$$
for $j=1,\dots,n-1$ that forgets all but the $j$th and $n$th points. As ${\pi_j}_*[F]=[F]$ is an extremal nef curve and the pushforward of a nef curve is nef, if $[F']$ is a nef curve class appearing in a nef decomposition of $[F]$ then ${\pi_j}_*[F']=k_j[F]$ for some $0\leq k_j\leq 1$. This provides enough linearly independent conditions to show that no nontrivial nef decomposition exists and indeed $[F]$ is an extremal nef curve class.

\begin{ack}
I am grateful to Dawei Chen and Martin M\"{o}ller for useful discussions related to this paper and I thank the anonymous referees for many useful comments and improvements. The author was supported by the Alexander von Humboldt Foundation during the preparation of this article.
\end{ack}
\section{Intersection theory on $\Mbar{g}{n}$}
In this section we present a number of preliminaries on the intersection theory on $\Mbar{g}{n}$ followed by new results to be used in the following sections.
\subsection{The Picard group of $\Mbar{g}{n}$}\label{Pic}
The Picard group $\Pic(\Mbar{g}{n})\otimes\RR=N^1(\Mbar{g}{n})$ is generated by $\lambda$, the first Chern class of the Hodge bundle, $\psi_i$ the first Chern class of the cotangent bundle on $\Mbar{g}{n}$ associated with the $i$th marking for $i=1,\dots,n$ and the classes of the irreducible components of the boundary. We denote by $\delta_{0}$ the class of the locus of curves with a non-separating node and $\delta_{i:S}$ for $0\leq i\leq g$, $S\subset\{1,\dots,n\}$ the class of the locus of curves with with a separating node, such that one component has genus $i$ and contains precisely the markings of $S$. Hence $\delta_{i:S}=\delta_{g-i:S^c}$ and we require $|S|\geq 2$ for $i=0$ and $|S|\leq n-2$ for $i=g$. For $g\geq 3$ these divisors freely generate $\Pic(\Mbar{g}{n})\otimes\RR$. When $g=2$ the classes $\lambda$, $\delta_0$ and $\delta_1$ generate $\Pic(\overline{\mathcal{M}}_2)\otimes\RR$ with the relation
\begin{equation*}
\lambda=\frac{1}{10}\delta_0+\frac{1}{5}\delta_1.
\end{equation*}
Pulling back this relation under the forgetful morphism forgetting all marked points gives the only relation on these generators in $\Pic(\overline{\mathcal{M}}_{2,n})\otimes \RR$. See~\cite{AC} or \cite{HarrisMorrison} for an introduction to the divisor theory of $\Mbar{g}{n}$.

\subsection{Cones of cycles and duality}
For a projective variety $X$, let $N^1(X)$ denote the $\RR$-vector space of divisors modulo numerical equivalence and $N_1(X)$ the $\RR$-vector space of curves modulo numerical equivalence. The cycles in $N^1(X)$ that can be written as a positive sum of effective divisors form a convex cone in $N^1(X)$ known as the \emph{effective cone} or \emph{effective cone of divisors} denoted $\text{Eff}^1(X)$ the closure of this cone is known as the \emph{pseudoeffective cone} or \emph{pseudoeffective cone of divisors} and is denoted $\overline{\text{Eff}}^1(X)$. The \emph{effective cone of curves} and \emph{pseudoeffective cone of curves} are defined similarly and denoted by $\text{Eff}_1(X)$ and $\overline{\text{Eff}}_1(X)$ respectively. 

A pseudoeffective divisor that has non negative intersection with all effective curves is known as nef and the closed cone of all such divisors forms a subcone of $\overline{\text{Eff}}^1(X)$ we denote by $\overline{\text{Nef}}^1(X)$.  The cone of nef curves $\overline{\text{Nef}}_1(X)$ is similarly defined as the pseudoeffective curves with non-negative intersection with all pseudoeffective divisors. 

The nef cone of curves $\overline{\text{Nef}}_1(X)$ is dual to the pseudoeffective cone of divisors $\overline{\text{Eff}}^1(X)$ while the nef cone of divisors $\overline{\text{Nef}}^1(X)$ is dual to the pseudoeffective cone of curves $\overline{\text{Eff}}_1(X)$.

\subsection{Rigitity, extremality and base loci}
A cycle $[E]$ is \emph{extremal} in a cone if it cannot be written as a sum 
$$[E]=a[A]+b[B]$$
for $a,b>0$ of cycles $[A]$ and $[B]$ also lying in the cone unless $[E],[A]$ and $[B]$ are all proportional classes. An effective cycle $E$ is \emph{rigid} if every cycle with class $m[E]$ for $m>0$ is supported on $E$. For example, an effective divisor $D$ is rigid if 
$$\dim H^0(X, mD)=1$$
for all $m>0$.

If irreducible curves with numerical class equal to $B$ cover a Zariski  dense subset of an irreducible effective divisor $D$, then $B$ is known as a \emph{covering curve} for $D$. The following lemma~\cite[Lemma 4.1]{ChenCoskun} provides a well-known trick for showing a divisor is rigid and extremal.
\Covering*

The utility of covering curves is not restricted to showing extremality of irreducible effective divisors. Covering curves can also be used to identify rigid divisors in the base locus of the linear system associated to other effective divisors.

\begin{restatable}{lem}{BaseLocus}
\label{lemma:base}
If $D$ is an effective divisor and $[B]$ is a covering curve class for an irreducible effective divisor $E$ with $[B]\cdot[E]<0$ and $[B]\cdot [D]<0$ then for any $k>0$,
$$|kD|=\left|k\left([D]-\frac{[B]\cdot [D]}{[B]\cdot[E]}[E]\right)\right|+   k\frac{[B]\cdot [D]}{[B]\cdot[E]}E $$
\end{restatable}

\begin{proof}
As irreducible curves with class equal to $[B]$ cover a Zariski dense subset of $E$, if $[B]\cdot [D]<0$ for any effective divisor $D$ then $|D|$ must contain an unmovable component supported on $E$.
\end{proof}

The setting of $\Mbar{g}{2}$ offers a method of constructing a potential covering curve for any effective divisor.

\begin{defn}
Let $C$ be a general smooth genus $g$ curve and
$$i:C\times C\longrightarrow \Mbar{g}{2}$$
the natural morphism. For a fixed divisor $[D]$ in $\Mbar{g}{2}$, define
$$[B_D]:=i_*i^*[D].$$
\end{defn}

\begin{prop}\label{Thm:eff}
If $D$ is effective, then $[B_D]$ is effective or zero
\end{prop}
\begin{proof}
As the curve $C$ is general, the pullback $i^*[D]$ of an effective divisor $D$ is either effective or zero.
\end{proof}

\begin{prop}\label{prop:BDD}Let
$$[D]=c_{\psi_1}\psi_1+c_{\psi_2}\psi_2+c_{0:\{1,2\}}\delta_{0:\{1,2\}}+c_\lambda\lambda+c_0\delta_0+\sum_{i=1}^{g-1}c_{i:\emptyset}\delta_{i:\emptyset}+\sum_{i+1}^{g-1}c_{i:\{1\}}\delta_{i:\{1\}}$$
then
$$[B_D]\cdot [D]=(2g-2)\left((4g-4)c_{\psi_1}c_{\psi_2}+(c_{\psi_1}+c_{\psi_2})^2-c_{0:\{1,2\}}^2  \right)$$
\end{prop}
\begin{proof} Observe that the pullback of all standard generators of $\Pic(\Mbar{g}{2})\otimes\RR$ are zero other than $\psi_1,\psi_2$ and $\delta_{0:\{1,2\}}$. Hence
$$[B_D]\cdot [D]=i_*i^*[D]\cdot[D]=i^*[D]\cdot i^*[D]=(c_{\psi_1}i^*\psi_1+c_{\psi_2}i^*\psi_2+c_{0:\{1,2\}}i^*\delta_{0:\{1,2\}})^2$$
If we let $\Delta\subset C\times C$ be the diagonal and $f_i$ for $i=1,2$ be the numerical class of the fibre of the projection of $C\times C$ onto the $i$th component. Then
$$ f_i\cdot f_i=0,\hspace{0.7cm} f_i\cdot\Delta=1,\hspace{0.7cm} f_1\cdot f_2=1,\hspace{0.7cm} \Delta^2=2-2g$$
and further by standard computation of the intersections of the curves $i_*f_1, i_*f_2$ and $i_*\Delta$ with divisors $\psi_1,\psi_2$ and $\delta_{0:\{1,2\}}$ we obtain
$$i^*\psi_i=(2g-2)f_i+\Delta,\hspace{0.7cm} i^*\delta_{0:\{1,2\}}=\Delta.$$
Hence the result holds.
\end{proof}
Restricting to irreducible divisors we obtain the following.
\CoveringExtremal*

\begin{proof}
If $E$ is an irreducible effective divisor and $[B_E]\cdot [E]<0$ then $[B_E]$ is effective by Proposition~\ref{Thm:eff}. If for a general curve $C$, the effective curve $i_*i^*D$ is reducible, then each component has the same class. Otherwise the components would be distinguishable, and taking the closure of each of the component curves over varying general $[C]\in \M{g}{}$ we would obtain multiple connected components of $E$ contradicting the irreducibility of $E$. By Lemma~\ref{lemma:covering} we have $E$ is rigid and extremal.
\end{proof}
We include a couple of simple examples of applications of this proposition, both already known to be rigid and extremal.

\begin{ex}\label{ex1}
Consider the irreducible effective divisor $\delta_{0:\{1,2\}}$. Observe
$$ [B_{\delta_{0:\{1,2\}}}]\cdot \delta_{0:\{1,2\}}= 2-2g$$
and hence by Proposition~\ref{Prop:Extremal} this divisor is rigid and extremal.
\end{ex}

\begin{ex}\label{ex2}
Consider the irreducible effective divisor $H$ in $\Mbar{2}{2}$ obtained as the closure of the locus of $[C,p_1,p_2]\in \M{2}{2}$ such that $p_1$ and $p_2$ are conjugate under the unique hyperelliptic involution on $C$. The class of $H$ is known to be (see~\cite{HarrisMorrison})
$$[H]=\psi_1+\psi_2-\lambda -3\delta_{0:\{1,2\}} -\delta_{1:\emptyset}  $$
giving
$$ [B_{H}]\cdot [H]=-2 $$
and hence by Proposition~\ref{Prop:Extremal} this divisor is rigid and extremal. 
\end{ex}

The utility of the construction of these covering curves in $\Mbar{g}{2}$ will be in using such curves to identify base loci in the linear system associated to other effective divisors. Appealing to Lemma~\ref{lemma:base} we obtain.

\BaseLociBD*

\begin{proof}
Consider an effective divisor $D$ such that $[B_D]\cdot [D]<0$. Consider an effective decomposition of $[D]$ into irreducible distinct divisors $D_i$
$$[D]=\sum c_i[D_i]$$
with $c_i>0$. Then 
$$[B_D]=\sum c_i [B_{D_i}].$$
But $[B_D]\cdot [D]<0$ implies $[B_{D_i}]\cdot [D]<0$ for some $D_i$, which we label $E$. Further, necessarily $[B_{E}]\cdot [E]<0$ and hence $E$ is rigid and extremal. Applying Lemma~\ref{lemma:base} completes the proof.
\end{proof}

\section{The rank of the pseudoeffective dual space}\label{DualSpace}
Let $F$ be a general fibre of the morphism $\pi:\Mbar{g}{n}\longrightarrow \Mbar{g}{n-1}$ that forgets the $n$th marked point. Irreducible curves with class equal to $[F]$ cover an open dense subset of $\Mbar{g}{n}$. Any effective divisor with negative intersection with $[F]$ would need to contain this dense subset and hence $[F]$ is a nef curve and must have nonnegative intersection with every pseudoeffective divisor. For a nef curve class $[B]$ we define the pseudoeffective dual space
$$[B]^\vee:=\{[D]\in \overline{\text{Eff}}^1(\Mbar{g}{2})\hspace{0.2cm}\big|\hspace{0.2cm} [B]\cdot[D]=0\}.$$
In this section we prove the following theorem.

\Dual*
 
 \begin{proof}
The case $n=2$ is Proposition~\ref{Prop:n2} and $n\geq3$ is Proposition~\ref{Prop:n3}.
 \end{proof}
 
First, we consider strictly effective divisors.

 \begin{prop}\label{Prop:eff}For $g\geq 2$ and $n\geq 1$
 $$\{[D]\in{\text{Eff}}(\Mbar{g}{n})\hspace{0.2cm}\big| \hspace{0.2cm}[F]\cdot [D]=0\}\otimes \RR=\RR\langle\{\psi_i-\delta_{0:\{i,n\}}\}_{i=1}^{n-1},\lambda,\delta_0, \{\delta_{i:\emptyset}\}_{i=1}^{g-1},\{ \delta_{i:\{1\}}  \}_{i=1}^{g-1}  \rangle.$$
 \end{prop}
  
 \begin{proof}
Consider an irreducible effective divisor $D$ in $\Mbar{g}{n}$ that intersects the interior $\M{g}{n}$ with $[F]\cdot [D]=0$. For any point $[C,p_1,\dots,p_n]\in \M{g}{n}$ in the support of $D$, consider the curve $B$ equal to the fibre of $[C,p_1,\dots,p_{n-1}]$ under the morphism $\pi:\Mbar{g}{n}\longrightarrow \Mbar{g}{n-1}$ that forgets the $n$th marked point. As $[B]=[F]$ we have $[B]\cdot[D]=0$ and necessarily, the irreducible curve $B$ is contained in the support of $D$. Hence the divisor $D$ is a pullback of an effective divisor under $\pi$. Further, if $D$ is an irreducible divisor in $\Mbar{g}{n}$ that does not intersects the interior $\M{g}{n}$ with $[F]\cdot [D]=0$ then $D$ is an irreducible boundary component other than $\delta_{0:\{i,n\}}$ for $i=1,\dots,n-1$. Hence the result holds.
\end{proof}

Hence $[F]^\vee\otimes\RR$ has rank $\geq \rho(\Mbar{g}{1})-n$ which for $n=1$ shows $[F]^\vee\otimes\RR$ has rank $\rho(\Mbar{g}{1})-1$. For $n\geq 2$ to extend this statement to the theorem we are left to consider the case that $[D]$ is a strictly pseudoeffective divisor
$$[D]\in \overline{\text{Eff}}(\Mbar{g}{n})\setminus {\text{Eff}}(\Mbar{g}{n})$$
with $[F]\cdot[D]=0$ and $[D]$ lies outside the space detailed above. We consider first the case $n=2$.

\begin{prop}\label{Prop:n2}
For $g\geq2$ and $n=2$
$$\text{rank}([F]^\vee\otimes\RR)=\rho(\Mbar{g}{n})-2.    $$
\end{prop}

\begin{proof}
Let $[D]$ be a pseudoeffective divisor with $[F]\cdot [D]=0$ outside of the space detailed in Proposition~\ref{Prop:eff}. The class of $[D]$ must be of the form
$$[D]=((1-2g)c_{\psi_2}-c_{0:\{1,2\}})\psi_1+c_{\psi_2}\psi_2+c_{0:\{1,2\}}\delta_{0:\{1,2\}}+c_{\lambda}\lambda+c_0\delta_0+\sum_{i=1}^{g-1}c_{i:\{1\}}\delta_{i:\{1\}}+\sum_{i=1}^{g-1}c_{i:\emptyset } \delta_{i:\emptyset }$$
with $c_{\psi_2}\ne 0$. But then
$$[B_D]\cdot [D]=-8c_{\psi_2}^2(g-1)^2g<0.$$
Now consider effective divisor classes limiting to the pseudoeffective class $[D]$, that is, an $\RR$-continuous function (defined via the Euclidean metric on the coefficients of divisors classes in some fixed basis) of effective classes $[D_t]$ for $t>0$ such that 
$$[D]=\lim_{t\to 0^+}[D_t].$$ 
Then as the intersection product is continuous, there exists some $\varepsilon>0$ such that $[B_{D_\varepsilon}]\cdot[D_t]<0$ for $0\leq t\leq \varepsilon$. But this implies a common irreducible rigid and extremal component, say $E$ such that $[B_E]\cdot[E]<0$ and $[B_E]\cdot[D_t]<0$ for all $0\leq t\leq \varepsilon$. Hence by Proposition~\ref{Prop:BaseLociBD} 
$$[D_t]-\frac{[B_E]\cdot[D_t]}{[B_E]\cdot[E]}[E]$$
is effective for $0< t\leq \varepsilon$ implying that the limit
$$[D]-\frac{[B_E]\cdot[D]}{[B_E]\cdot[E]}[E]$$
is pseudoeffective. However, $E$ is effective and $[B_E]\cdot [E]<0$, hence $[E]$ lies outside the space of effective divisors with zero intersection with $[F]$ specified in Proposition~\ref{Prop:eff} and hence $[F]\cdot[E]>0$. But $[F]\cdot [D]=0$ providing a contradiction to the existence of such a pseudoeffective class. Hence for $n=2$ the rank of $[F]^\vee\otimes\RR$ is $\rho(\Mbar{g}{2})-2$.
\end{proof}

This result can be extended to the case $n\geq3$ through the use of gluing morphisms, however, we will need one technical result on extremal pseudoeffective rays. Let $||\cdot||$ be the Euclidean metric on divisors in $\Pic(\Mbar{g}{n})\otimes\RR$ expressed in the standard basis.

\begin{lem}\label{Lem:ps}
Fix $g\geq2$ and $n\geq3$. If $[D]\ne k\delta_{0:\{1,\dots,n-1\}}$ for $k>0$ is extremal in $\overline{\text{Eff}}^1(\Mbar{g}{n})$, then for every $\varepsilon>0$ there exists an effective divisor $[G_\varepsilon]$ with an effective decomposition not supported on $\delta_{0:\{1,\dots,n-1\}}$ such that $||[D]-[G_\varepsilon]||<\varepsilon$.
\end{lem}

\begin{proof}
If $[D]$ is effective then let $[D]=[G_\varepsilon]$ for all $\varepsilon>0$. If $[D]$ is strictly psedoeffective then for every $\epsilon>0$ there exists a $\delta(\epsilon)>0$ such that for all $[H]\in\text{Eff}^1(\Mbar{g}{n})$ if
$$||[D]-[H]||<\delta(\epsilon)$$
then $[H]-\epsilon\delta_{0:\{1,\dots,n-1\}}$ is not effective. If not, then $[D]-\alpha\delta_{0:\{1,\dots,n-1\}}$ is pseudoeffective for some $\alpha>0$ contradicting the assumption that $[D]$ is extremal.

Now for any $\varepsilon>0$ set $\epsilon=\frac{\varepsilon}{2}$ and consider the non-empty set of divisors divisors $[H]\in\text{Eff}^1(\Mbar{g}{n})$ such that 
$$||[D]-[H]||<\min(\delta(\epsilon),\epsilon).$$
For any such $[H]$ there exists an effective decomposition into irreducible effective divisors with coefficient of $\delta_{0:\{1,\dots,n-1\}}$ less that $\epsilon$. Removing this term gives an effective divisor $[\overline{H}]$ not supported on $\delta_{0:\{1,\dots,n-1\}}$ and by the triangle inequality
$$||[D]-[\overline{H}]||<\epsilon+\epsilon=\varepsilon.$$
Setting $[G_\varepsilon]=[\overline{H}]$ concludes the proof.
\end{proof}

We proceed with this result to prove the remaining cases of Theorem~\ref{Thm:dual}.

\begin{prop}\label{Prop:n3}
For $g\geq2$ and $n\geq3$
$$\text{rank}([F]^\vee\otimes\RR)\leq \rho(\Mbar{g}{n})-2.    $$
\end{prop}

\begin{proof}
For a fixed general $[\PP^1,q,p_1,\dots,p_{n-1}]\in \mathcal{M}_{0,n}$ consider the map 
\begin{eqnarray*}
\begin{array}{cccc}
\alpha:&\overline{\mathcal{M}}_{g,2}&\rightarrow& \overline{\mathcal{M}}_{g,n}\\
&[C,q_1,q_2]&\mapsto&[C\bigcup_{q_1=q}\PP^1,p_1,p_2,\dots,p_{n-1},q_2].
\end{array}
\end{eqnarray*} 
that glues points $q_1$ and $q$ to form a node. The pullback of the generators of the Picard group are~\cite{AC}
\begin{equation*}
\alpha^*\lambda=\lambda, \hspace{0.3cm}\alpha^*\delta_0=\delta_0,\hspace{0.3cm}\alpha^*\delta_{0:\{1,\dots,n-1\}}=-\psi_1,\hspace{0.3cm}\alpha^*\delta_{0:\{1,\dots,n\}}=\delta_{0:\{1,2\}},\hspace{0.3cm}  \alpha^*\psi_i=\begin{cases}
0 &\text{ for $i=1,\dots,n-1$}\\
\psi_2 &\text{ for $i=n$}
\end{cases}
\end{equation*}
and for $1\leq i\leq g-1$
\begin{equation*}
\alpha^*\delta_{i:\{n\}}=\delta_{i:\{n\}}, \hspace{0.3cm}\alpha^*\delta_{i:\emptyset}=\delta_{i:\emptyset}
\end{equation*}
and the pullback of all other boundary divisors is zero.

For $n\geq 3$ let $[D]$ be the class of an extremal strictly pseudoeffective divisor such that $[F]\cdot[D]=0$. 
Observe that the pullback under $\alpha$ of any effective divisor not supported on $\delta_{0:\{1,\dots,n-1\}}$ is effective and hence by Lemma~\ref{Lem:ps}, $\alpha^*[D]$ is pseudoeffective.

Though the fibre of the morphism $\pi: \Mbar{g}{n}\longrightarrow\Mbar{g}{n-1}$ over a general point is irreducible, for $n\geq3$ the fibre over a general point in $\delta_{0:\{1,\dots,n-1\}}$ has two reducible components and hence
$$[F]=[B_1]+[B_2]$$
where
$$[B_1]\cdot \psi_n=2g-1,\hspace{0.5cm}[B_1]\cdot\delta_{0:\{1,\cdots,n-1\}}=-1,\hspace{0.5cm}[B_1]\cdot\delta_{0:\{1,\cdots,n\}}=1$$
and
$$[B_2]\cdot \psi_n=n-2,\hspace{0.8cm}[B_2]\cdot \psi_i=[B_2]\cdot\delta_{0:\{i,n\}}=1 \text{ for $i=1,\dots,n-1$},$$
$$[B_2]\cdot\delta_{0:\{1,\cdots,n-1\}}=1,\hspace{0.8cm}[B_2]\cdot\delta_{0:\{1,\cdots,n\}}=-1$$
with all other intersections equal to zero. Further, $[B_1]$ forms a covering curve for $\delta_{0:\{1,\cdots,n-1\}}$ as clearly by varying the fibre, irreducible curves with class equal to $[B_1]$ cover an open Zariski dense subset of $\delta_{0:\{1,\cdots,n-1\}}$. Similarly, $[B_2]$ forms a covering curve for $\delta_{0:\{1,\cdots,n\}}$. 

It follows that if $[B_1]\cdot[D]<0$ then $\delta_{0:\{1,\dots,n-1\}}$ will form a rigid component of all effective divisors near $[D]$ and hence by the argument used above in the $n=2$ case
$$[D]+([B_1]\cdot[D])\delta_{0:\{1,\dots,n-1\}}$$
is a pseudoeffective divisor contradicting the assumption that $[D]$ is extremal. Similarly if $[B_2]\cdot[D]<0$ then
$$[D]+([B_2]\cdot[D])\delta_{0:\{1,\dots,n\}}$$
 is pseudoeffective, again, giving a contradiction. Hence as $[F]=[B_1]+[B_2]$ and $[F]\cdot[D]=0$ we have
 $$[B_1]\cdot[D]=[B_2]\cdot[D]=0.$$
 Now observe that for $[F]\in \overline{\text{Nef}}_1(\Mbar{g}{2})$ we have $\alpha_*[F]=[B_1]$ and hence
 $$[F]\cdot\alpha^*[D]=\alpha_*[F]\cdot [D]=[B_1]\cdot [D]=0. $$
 Hence in $\Pic_\QQ(\Mbar{g}{2})$ we have $\alpha^*[D]\in [F]^\vee$ which pulling back the class of $[D]$ implies $c_{\psi_n}=0$. But this gives two linearly independent equations on $[F]^\vee\otimes \RR$ in $\Pic(\Mbar{g}{n})$,
 \begin{eqnarray*}
 c_{\psi_n}=0&&\\
 \sum_{i=1}^{n-1}(c_{\psi_i}+c_{0:\{i,n\}})=0.&&
 \end{eqnarray*}
Hence the corank is at least $2$.
 \end{proof}

\section{Divisors from Hurwitz spaces}
In this section we define two families of effective divisors in $\Mbar{g}{2}$ via Hurwitz spaces and compute their class. These families of divisors will be used asymptotically to show that $[F]$ cannot have a nef decomposition and is hence extremal. The first divisor is the closure of the locus $D_k$ for fixed $k\geq 2$ of $[C,p_1,p_2]$ in $\M{g}{2}$ such that $C$ admits a degree $k(g-1)+1$ cover of a rational curve with one fibre containing only $p_1$, that is $p_1$ is a ramification point of the cover with order $k(g-1)$ and a second fibre containing an unramified point $p_2$ and $g-1$ other points that are all ramified with order $k-1$. Hence
\begin{equation*}
D_k:={\left\{ [C,p_1,p_2]\in\M{g}{2} \big| \exists [C,p_1,p_2,q_1,\dots,q_{g-1}]\in\M{g}{g+1}, \OO_C((k(g-1)+1)p_1-p_2-k\sum_{i=1}^{g-1}q_i)\sim\OO_C   \right\}}.
\end{equation*}
The second family of divisors is the closure of the locus $E_k$ for fixed $k\geq 2$ of $[C,p_1,p_2]$ in $\M{g}{2}$ such that $C$ admits a degree $k(g-1)$ cover of a rational curve with one fibre consisting of an unramified point $p_2$, a point $p_1$ ramified with order $k(g-1)-2$ and another fibre consisting of $g-1$ points each ramified with order $k-1$. Hence
\begin{equation*}
E_k:={\left\{ [C,p_1,p_2]\in\M{g}{2} \big| \exists [C,p_1,p_2,q_1,\dots,q_{g-1}]\in\M{g}{g+1}, \OO_C((k(g-1)-1)p_1+p_2-k\sum_{i=1}^{g-1}q_i)\sim\OO_C   \right\}}.
\end{equation*}
Taking the closure we obtain two families of divisors $\overline{D}_k$ and $\overline{E}_k$ for $k\geq 2$ in $\Mbar{g}{2}$. To compute the class of these divisors we define a number of test curves in $\Mbar{g}{2}$. Let $C$ be a general smooth genus $g$ curve and
$$i:C\times C\longrightarrow \Mbar{g}{2}$$
be the natural morphism. Again, let $\Delta\subset C\times C$ be the diagonal and $f_j$ for $j=1,2$ be the numerical class of the fibre of the projection of $C\times C$ onto the $j$th component. We define
$$F_j:=i_*f_j\hspace{0.7cm}B_\Delta=i_*\Delta$$
and obtain the following theorem.
\begin{thm}\label{Int}
The following intersection numbers hold
\begin{center}
\begin{tabular}{|p{0.5cm}|p{1cm}|p{1cm}|p{1cm}|p{5cm}|p{5cm}|  }
 \hline
  &$\psi_1$&$\psi_2$ &$\delta_{0:\{1,2\}}$ &$[\overline{D}_k]$&$[\overline{E}_k]$\\
 \hline
 $F_1$&$1$& $2g-1$ &$1$&$(k^{2 (g - 1)}-1) g $  &$(k^{2 (g - 1)}-1) g $ \\
$F_2$&$2g-1$ &$1$ &$1$&$(k (g - 1) + 1)^2 k^{2 (g - 1)} g - g$&$(k (g - 1) - 1)^2 k^{2 (g - 1)} g - g$  \\
$B_\Delta$&$0$&  $0$ &$2-2g$&$(k^{2 g} - 1) (g - 1)^{2}g + g (g^2 - 1)$& $(k^{2 g} - 1) (g - 1)^{2}g + g (g^2 - 1)$\\
\hline
\end{tabular}
\end{center}
and $F_1,F_2,B_\Delta$ have zero intersection with the other standard generators of $\Pic(\Mbar{g}{n})$.
\end{thm}
\begin{proof}
The first three columns are a well trodden exercise in intersection theory (See~\cite{HarrisMorrison}). We compute the final two columns via the Jacobian of a general genus $g$ curve. Consider the intersection $[F_1]\cdot [\overline{D}_k]$. Define
\begin{eqnarray*}
\begin{array}{cccccc}
f:C^g&\longrightarrow &\Pic^{k(g-1)+1}(C)\\
(p_2,q_1,\dots,q_{g-1})&\longmapsto&\OO_C(p_2+k\sum_{i=1}^{g-1}q_i).
\end{array}
\end{eqnarray*}
Investigating the fibre of this map above $[\OO_C((k(g-1)+1)p_1)]\in\Pic^d(C)$ for a general point $p_1$ will provide us the solutions of interest. The domain and range of $f$ are both of dimension $g$, we first compute the degree of the map $f$. Take a general point $e\in C$ and consider the isomorphism
\begin{eqnarray*}
\begin{array}{cccccc}
h:\Pic^{k(g-1)+1}(C)&\longrightarrow &J(C)\\
L&\longmapsto&L\otimes\OO_C(-({k(g-1)+1})e).
\end{array}
\end{eqnarray*}
Now let $F=h\circ f$. Then we have $\deg F=\deg f$. We observe
\begin{equation*}
F(p_2,q_1,\dots,q_{g-1})=\OO_C\biggl((p_2-e)+\sum_{i=1}^{g-1}k(q_i-e)\biggr).
\end{equation*}
Let $\Theta$ be the fundamental class of the theta divisor in $J(C)$. By \cite{ACGH} \S1.5 we have the locus of $\OO_C(k(x-e))$ for varying $x\in C$ has class $k^2\Theta$ in $J(C)$ and
\begin{equation*}
\deg \Theta^g=g!
\end{equation*}
Hence
\begin{eqnarray*}
\deg F&=&\deg F_*F^*([\OO_C])\\
&=&\deg \left(\Theta\prod_{i=1}^{g-1} k^2\Theta\right)\\
&=&g!k^{2(g-1)}
\end{eqnarray*}
Finally, there is one solution to omit as the solution $p_1=p_2=q_1=\dots =q_{g-1}$ will be counted in this value, but not contribute to $[F_2]\cdot[\overline{D}_k]$. To obtain the multiplicity of this solution, consider $F$ locally analytically around a point. If $f_0d\omega,...,f_{g-1}d\omega$ is a basis for $H^0(C,K_C)$, then locally analytically the map becomes
\begin{eqnarray*}
(p_2,q_1\dots,q_{g-1})&\longmapsto&\biggl(\int_e^{p_2}f_0d\omega+\sum_{i=1}^{g-1}k\int_e^{q_i}f_0d\omega,\dots,\int_e^{p_2}f_{g-1}d\omega+\sum_{i=1}^{g-1}k\int_e^{q_i}f_{g-1}d\omega\biggr)
\end{eqnarray*}
modulo $H_1(C,K_C)$. The map on tangent spaces at any fixed point $(p_2,q_1,\dots,q_{g-1})\in C^g$ is 
\begin{equation*}
DF_{(p_2,q_1,\dots,q_{g-1})}=\text{diag}(1,k,\dots,k)\begin{pmatrix}f_0(p_2)&f_0(q_1)&\dots&f_0(q_{g-1})\\
f_1(p_2)&f_1(q_1)&\dots&f_1(q_{g-1})\\
\dots&\dots&\dots&\dots&\\
f_{g-1}(p_2)&f_{g-1}(q_1)&\dots&f_{g-1}(q_{g-1})   \end{pmatrix}.
\end{equation*}
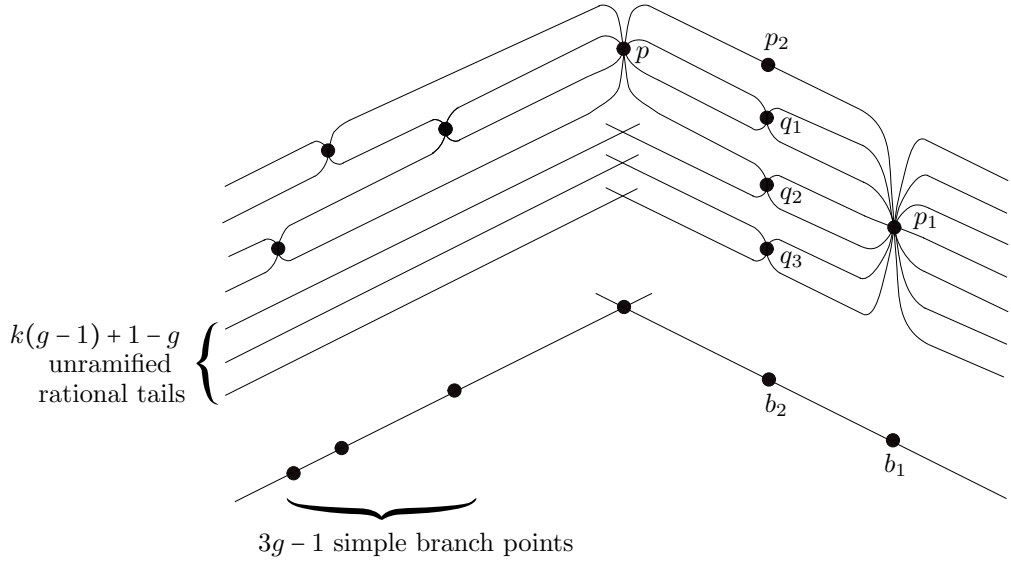
\begin{figure}[htbp]
\begin{center}
\begin{overpic}[width=0.8\textwidth]{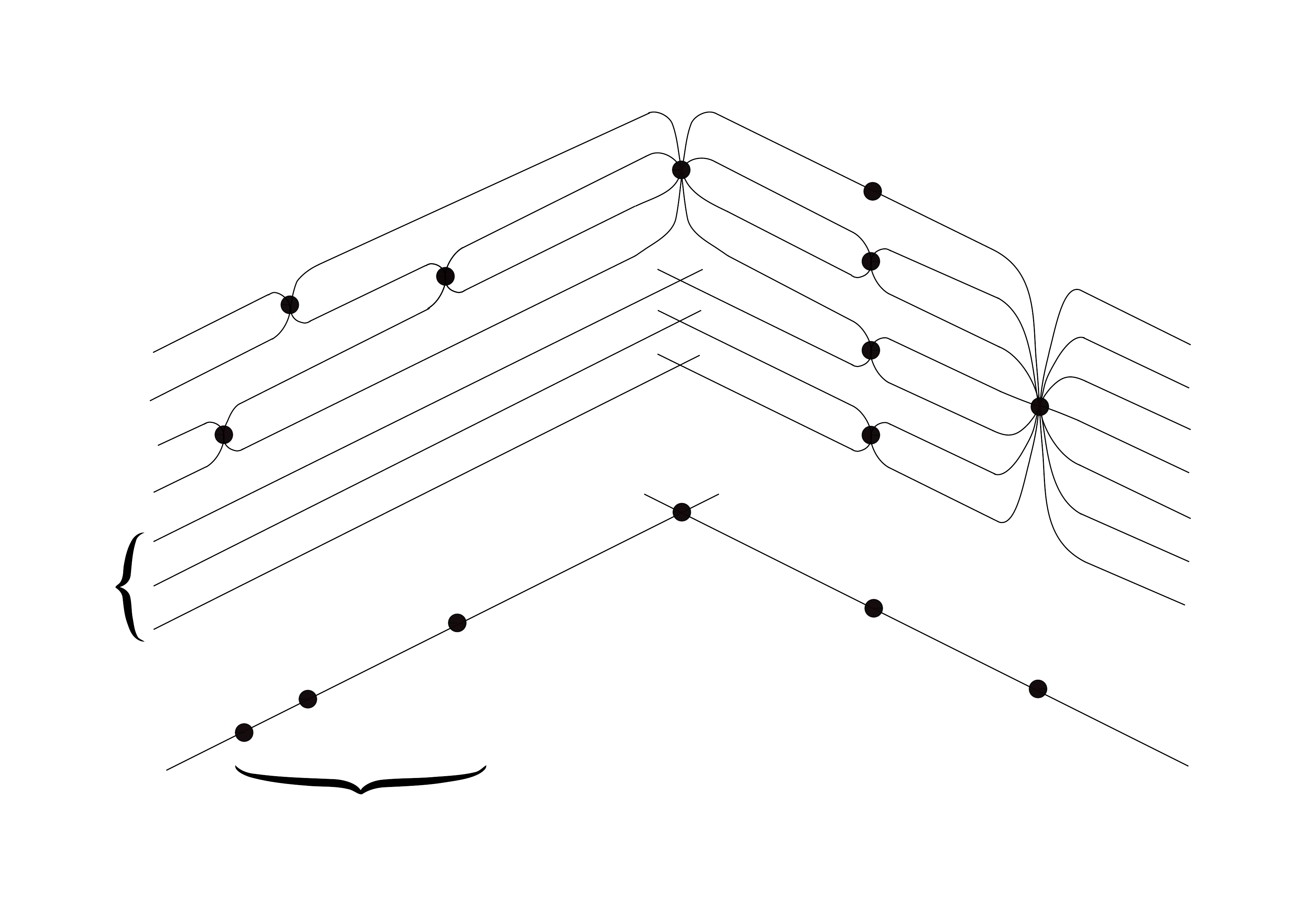}
\small{\put(15,6){$3g-1$ simple branch points}}
\put(-10 ,27){$k(g-1)+1-g$}
\small{\put(-7,24){ unramified}
\put(-8,21){ rational tails}}
\put(66,20){$b_2$}
\put(78,14){$b_1$}
\put(81,39){$p_1$}
\put(66,57){$p_2$}
\put(53,55.7){$p$}
\put(67.5,48.5){$q_1$}
\put(67.5,41.5){$q_2$}
\put(67.5,35){$q_3$}

 \end{overpic}
 \caption{$\overline{D}_k\cdot B_\Delta$ Case $1$: $b_1$ and $b_2$ colliding ($k=2$ and $g=4$). }
  \label{DkB1}

\end{center}
\end{figure}
Ramification in the map $F$ occurs when the map on tangent spaces is not injective which takes place at the points where $\rk(DF)<g$. When $p_1=p_2=q_1=\dots =q_{g-1}$, the map hence has ramification of order $g-1$ and the solution must be excluded with multiplicity $g$.
Allowing for the ordering of the $q_i$ we obtain
$$[F_2]\cdot [\overline{D}_k]=(k^{2(g-1)}-1)g.$$
Similarly, we obtain the entries $[F_1]\cdot [\overline{E}_k]$, $[F_2]\cdot [\overline{D}_k]$ and $[F_2]\cdot [\overline{E}_k]$.

To compute the intersections with $[B_\Delta]$ we require the input of the theory of admissible covers that compactifies spaces of branched covers. We require all admissible covers of the specified ramification profile such that after forgetting all but the source curve and $p_1$ and $p_2$, and performing stable reduction, we are left with the fixed general genus $g$ curve $C$ connected to a rational tail containing $p_1$ and $p_2$. The requirement that $C$ is a general curve restricts us, via a dimension count, to the case that precisely two branch points are colliding while the rest remain distinct. Let $b_1$ be the branch point that is the image of $p_1$ under the cover and let $b_2$ be the branch point that is the image of the $q_i$. Of the four possibilities, the only configurations of branch points colliding that will intersect $B_\Delta$ after stable reduction are when $b_1$ and $b_2$ collide, or when $b_1$ collides with a a simple branch point. 

\begin{figure}[htbp]
\begin{center}
\begin{overpic}[width=0.8\textwidth]{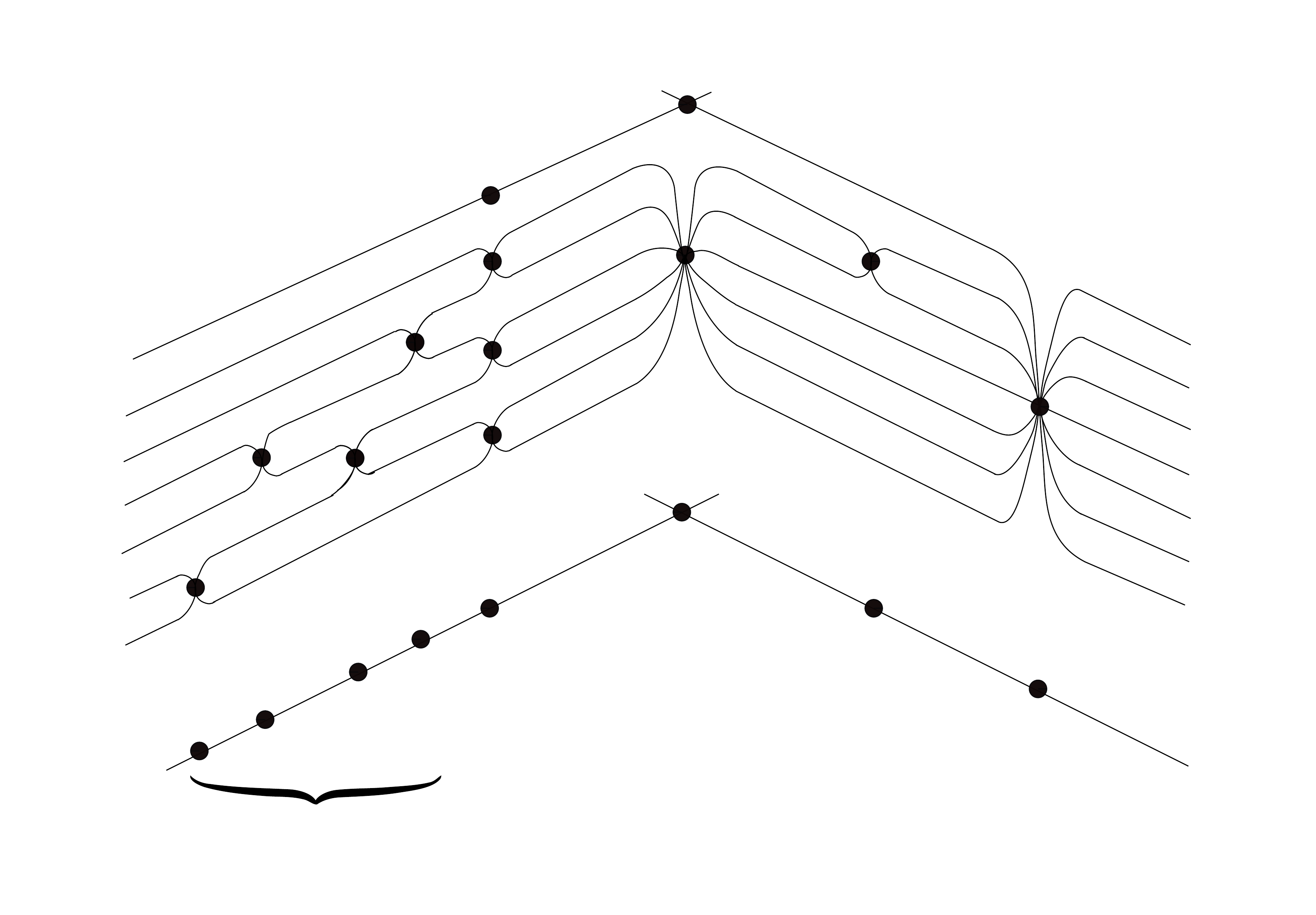}
\small{\put(12,6){$3g-2$ simple branch points}}

\small{\put(8,55){ unramified}
\put(7.25,52){ rational tail }
\put(8,49){containing $p_1$}}
\put(37,20){$b_2$}
\put(78,14){$b_1$}
\put(81,39){$p_1$}
\put(37,57){$p_2$}
\put(53.3,51.2){$p$}
\put(38.5,50){$q_1$}
\put(38.5,43){$q_2$}
\put(38.5,36.5){$q_3$}

\end{overpic}
 \caption{$\overline{D}_k\cdot B_\Delta$ Case $2$: $b_1$ colliding with a simple branch point ($k=2$ and $g=4$). }
  \label{DkB2}
\end{center}
\end{figure}
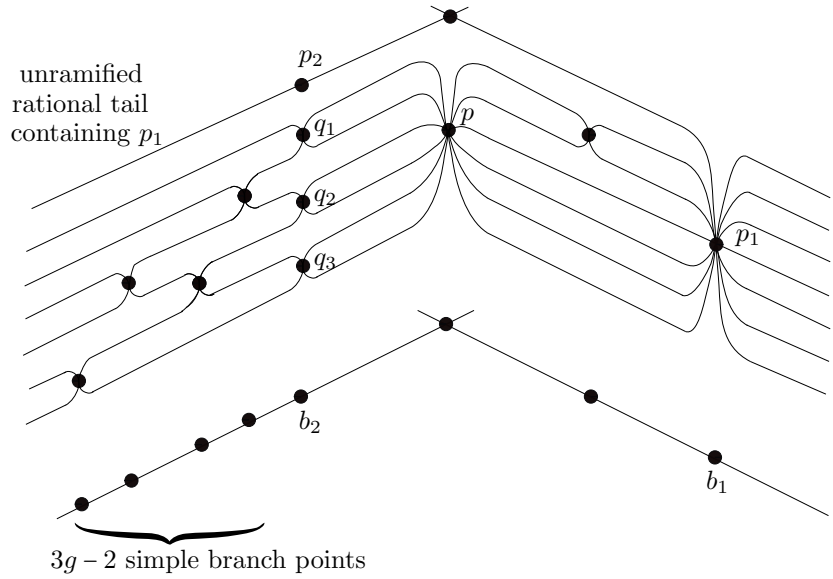

Consider first $[\overline{D}_k]\cdot [B_\Delta]$. The first case, when $b_1$ and $b_2$ come together is shown in Figure~\ref{DkB1}.  Riemann-Hurwitz shows that if the source curve of the component curve containing $p_1$ and $p_2$ has genus zero, then the ramification above the node must contribute
$$2(0)-2-(k(g-1)+1)(2(0)-2)-(k-1)(g-1)-k(g-1)=g-1.$$
In order for this to occur on a general curve $C$ we require that this ramification is concentrated at one point we label $p$ and the remaining $k(g-1)+1-g$ sheets of the cover are unramified rational tails as shown in Figure~\ref{DkB1}. 
But the requirement that $\dim H^0(gp)=2$ is equivalent to the requirement that $p$ is a Weierstrass point. Hence solutions of this type contribute $g(g^2-1)$ to the intersection $[\overline{D}_k]\cdot [B_\Delta]$. The same argument shows solutions of this type contribute $g(g^2-1)$ to the intersection $[\overline{E}_k]\cdot [B_\Delta]$. Figure~\ref{EkB1} shows intersections of this type.

\begin{figure}[htbp]
\begin{center}
\begin{overpic}[width=0.8\textwidth]{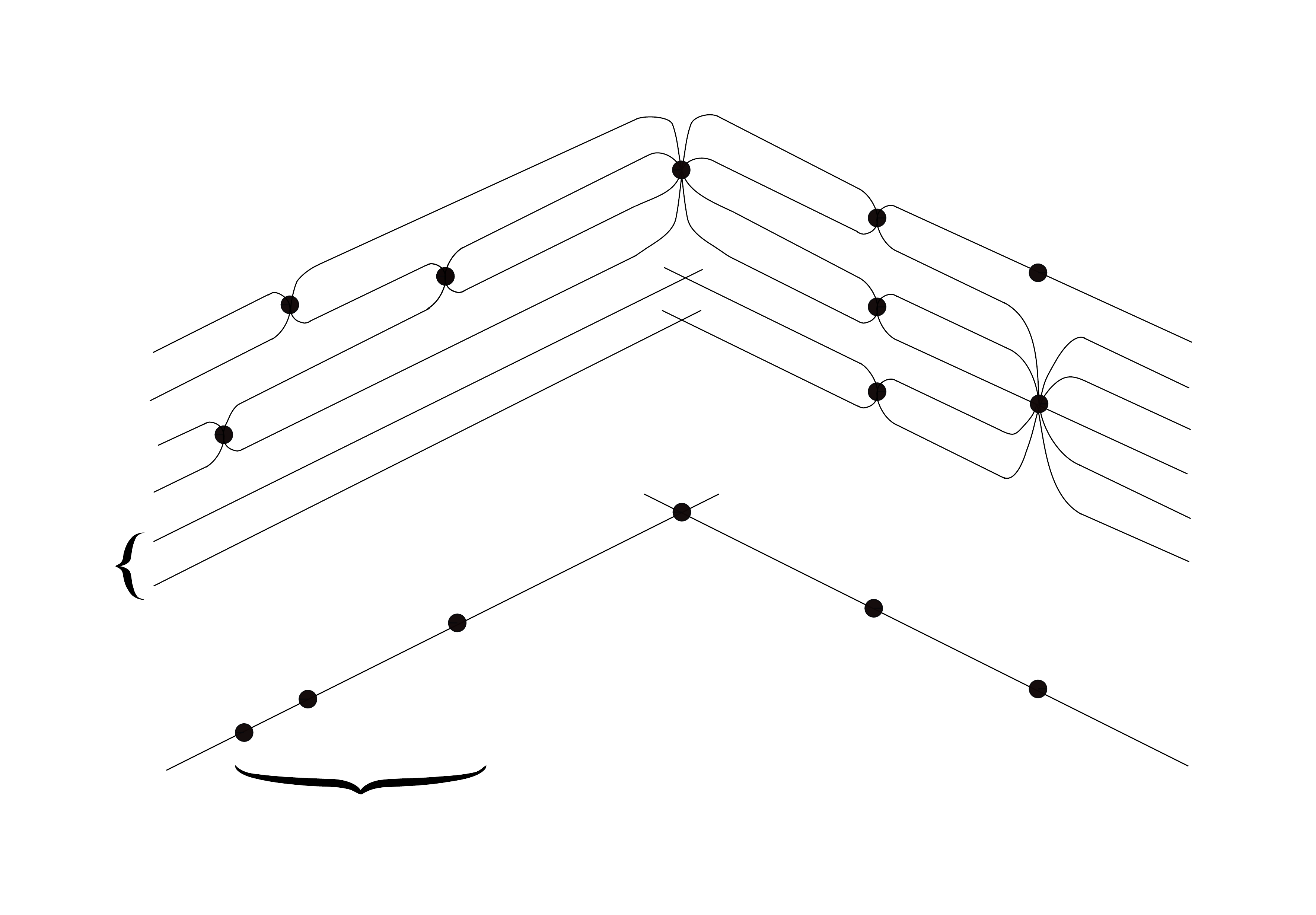}
\small{\put(15,7){$3g-1$ simple branch points}
\put(-5 ,27){$k(g-1)-g$}
\put(-5,24){ unramified}
\put(-6,21){ rational tails}}
\put(66,20){$b_2$}
\put(78,14){$b_1$}
\put(81,39){$p_1$}
\put(78,51){$p_2$}
\put(53,55.7){$p$}
\put(68,51.5){$q_1$}
\put(68,44.5){$q_2$}
\put(68,38){$q_3$}

\end{overpic}
 \caption{$\overline{E}_k\cdot B_\Delta$ Case $1$: $b_1$ and $b_2$ colliding ($k=2$ and $g=4$). }
  \label{EkB1}
\end{center}
\end{figure}
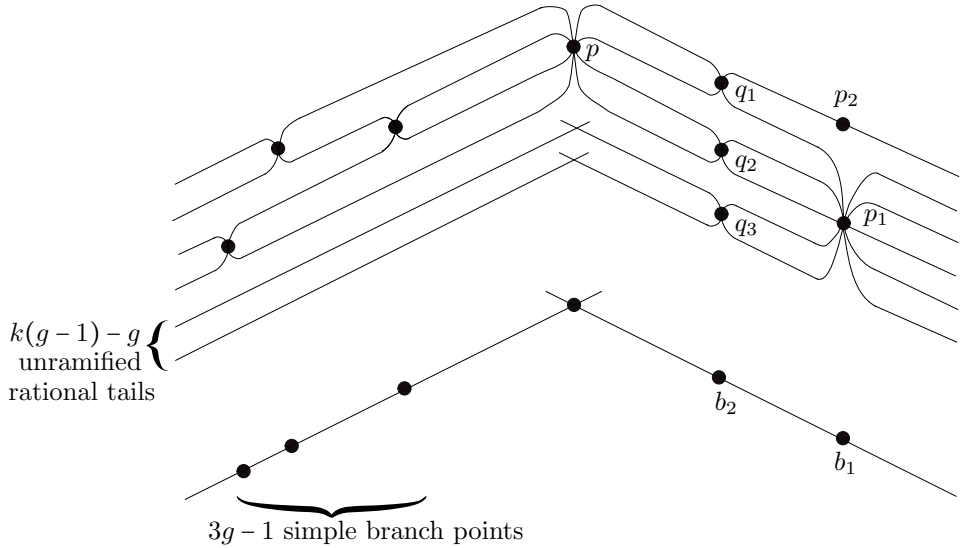

Consider now the second case contributing to the intersection $[\overline{D}_k]\cdot [B_\Delta]$ where $b_1$ comes together with a simple branch point. This case is shown in Figure~\ref{DkB2}. In order for the stable reduction to yield a rational tail containing $p_1$ and $p_2$ we require that $p_2$ sits on an unramified rational tail. Further, the requirement that $C$ is a general curve and a Riemann-Hurwitz computation yields that the remaining $k(g-1)$ sheets of the cover must come together at a single point $p$ above the node. Hence to enumerate such solutions we must find the points $p$ and $q_i$ on a general curve $C$ that provide primitive solutions to
$$\OO_C(k(g-1)p-k\sum_{i=1}^{g-1}q_i)\sim \OO_C.$$
Hence for $\eta_C\nsim \OO_C$ such that $\eta_C^{\otimes k}\sim \OO_C$ we require
$$\OO((g-1)p-\sum_{i=1}^{g-1}q_i)\sim\eta_C.$$
There are $k^{2g}-1$ such $\eta_C$ and hence by the same reasoning as previously, such solutions will contribute 
$$(k^{2g}-1)(g-1)^2g$$
to the intersection $[\overline{D}_k]\cdot [B_\Delta]$. 

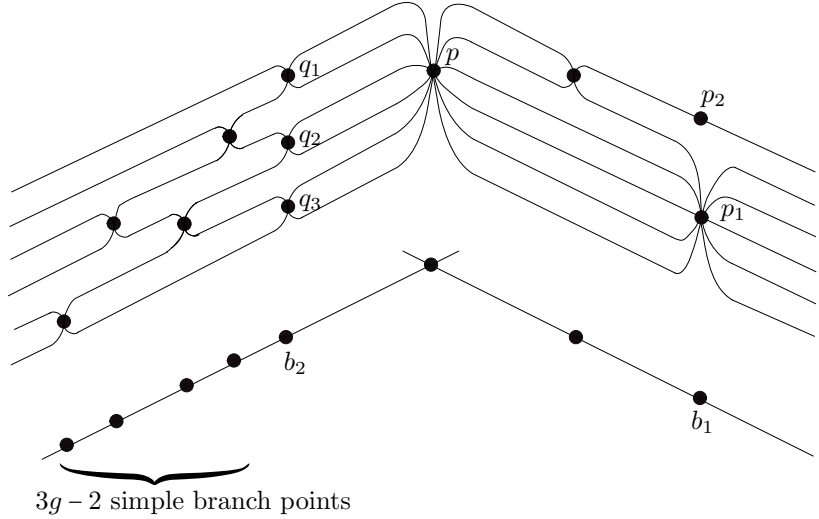
\begin{figure}[htbp]
\begin{center}
\begin{overpic}[width=0.8\textwidth]{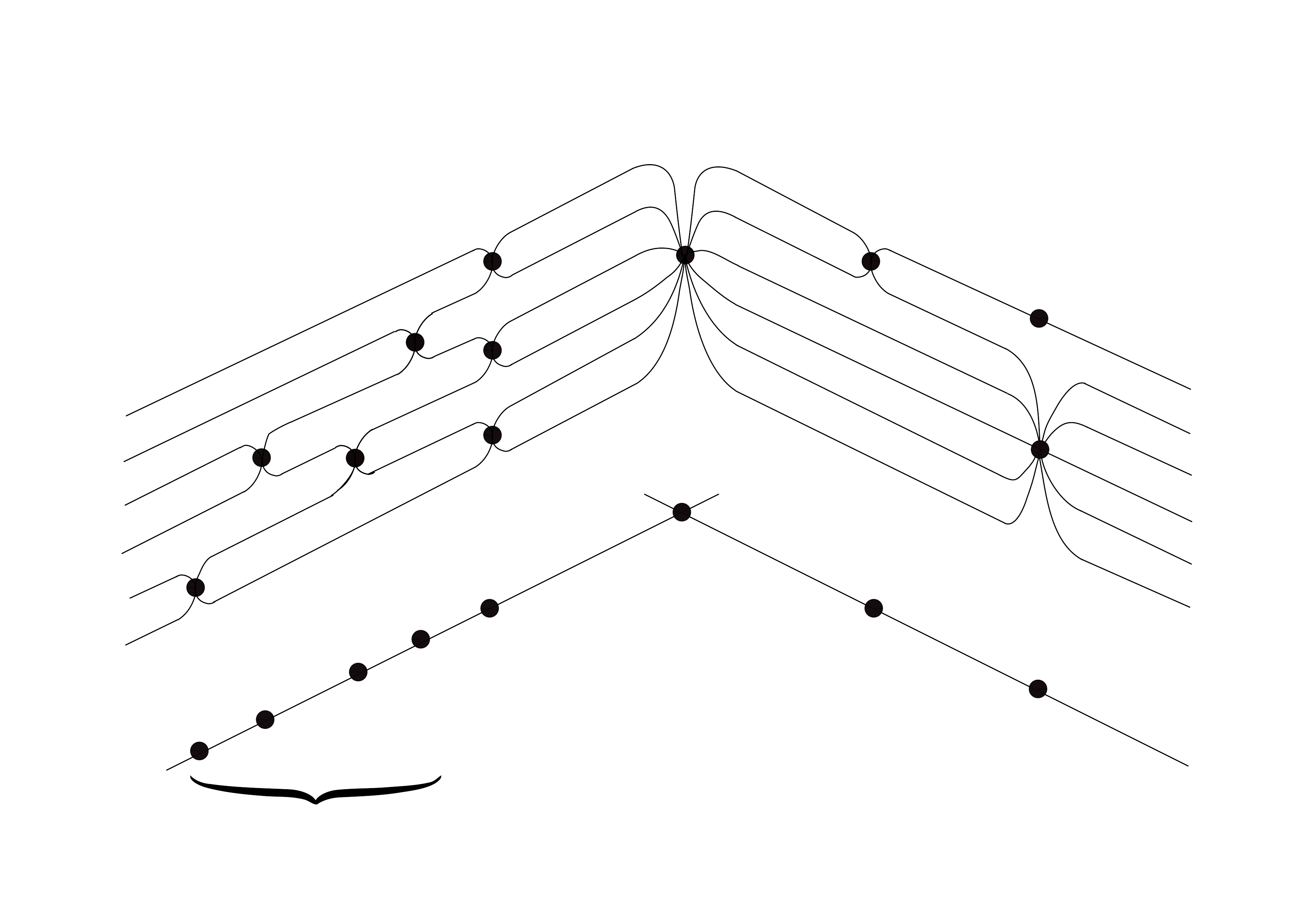}
\small{\put(12,6){$3g-2$ simple branch points}}

\put(37,20){$b_2$}
\put(78,14){$b_1$}
\put(81,35.5){$p_1$}
\put(79,47){$p_2$}
\put(53.3,51.2){$p$}
\put(38.5,50){$q_1$}
\put(38.5,43){$q_2$}
\put(38.5,36.5){$q_3$}
\end{overpic}
 \caption{$\overline{E}_k\cdot B_\Delta$ Case $2$: $b_1$ colliding with a simple branch point ($k=2$ and $g=4$). }
  \label{EkB2}

\end{center}
\end{figure}

The second case contributing to the intersection $[\overline{E}_k]\cdot [B_\Delta]$ where $b_1$ comes together with a simple branch point is shown in Figure~\ref{EkB2}. Riemann-Hurwitz shows that all sheets must come together at the point $p$ above the node and we are left with the same enumerative problem as the second case for the intersection $[\overline{D}_k]\cdot [B_\Delta]$. Hence we have solutions of this type contribute $(k^{2g}-1)(g-1)^2g$ to the intersection $[\overline{E}_k]\cdot [B_\Delta]$. This completes the proof.
\end{proof}

Theorem~\ref{Int} gives three independent relations in the coefficients of $\psi_1,\psi_2$ and $\delta_{0:\{1,2\}}$ and we have the following.

\begin{cor}\label{cor:class} The classes of divisors $\overline{D}_k$ and $\overline{E}_k$ in $\Pic(\Mbar{g}{2})$ are as follows
$$  [ \overline{D}_k]=\frac{1}{2}(gk+1)(gk-k+1)k^{2g-2}\psi_1+\frac{1}{2}(1-k)k^{2g-2}\psi_2 -\frac{1}{2}(gk^{2g}-k^{2g}+2)g\delta_{0:\{1,2\}}        +\dots$$
and
$$   [\overline{E}_k]=\frac{1}{2}(gk-1)(gk-k-1)k^{2g-2}\psi_1+\frac{1}{2}(k+1)k^{2g-2}\psi_2-\frac{1}{2}(gk^{2g}-k^{2g}+2)g\delta_{0:\{1,2\}}+\dots    $$
\end{cor}

\section{Extremality}
In this section we utilise the two interesting families of divisors defined in the last section and the other forgetful morphisms on $\Mbar{g}{n}$ to show that $[F]$ is indeed an extremal nef curve in $\overline{\text{Nef}}_1(\Mbar{g}{n})$.

\Extremal*

\begin{proof} For $n=1$ the rank of $[F]^\vee\otimes\RR$ is $\rho(\Mbar{g}{n})-1$ and hence $[F]$ is extremal.

Consider the case $n=2$. Any nef curve in a non-trivial nef decomposition of $[F]$ has zero intersection with 
$$\psi_1-\delta_{0:\{1,2\}},\lambda,\delta_0, \{\delta_{i:\emptyset}\}_{i=1}^{g-1},\{ \delta_{i:\{1\}}  \}_{i=1}^{g-1}. $$
Further, the intersection with $\psi_1$ and $\delta_{0:\{1,2\}}$ is positive as these classes are effective and zero intersection would imply a nef curve class with zero intersection with all standard generators of $\Pic(\Mbar{g}{2})$ except $\psi_2$. Such a curve class has negative intersection with the effective classes $\psi_2$ or $\overline{D}_k$ for any $k\geq 2$ and is hence not nef. 

Hence any nef curve in a non-trivial nef decomposition of $[F]$ has class proportional to $[F^t]$ for some $t\ne0$ where
$$[F^t]\cdot\psi_1=1,\hspace{1cm}[F^t]\cdot \psi_2=(2g-1)+t,\hspace{1cm}[F^t]\cdot\delta_{0:\{1,2\}}=1$$
and zero intersection with the other standard generators of $\Pic(\Mbar{g}{2})$.

Observe by Corollary~\ref{cor:class}, for fixed $t>0$ 
$$ [F^t]\cdot [\overline{D}_k]= (k^{2g - 2}-1) g +t\frac{1}{2}(1-k)k^{2g-2} =\frac{-t}{2}k^{2g-1}+\OO(k^{2g-2})<0 \text{  for  }k\gg0,$$
while for fixed $t<0$ 
$$ [F^t]\cdot [\overline{E}_k]=(k^{2g - 2}-1) g   +t \frac{1}{2}k^{2g-2}(k+1)=\frac{t}{2}k^{2g-1}+\OO(k^{2g-2}) <0 \text{  for  }k\gg0.$$
Hence for any $t\ne 0$ there exists an effective divisor with negative intersection with $[F^t]$ and hence $[F^t]$ is nef if and only if $t=0$. Hence $[F]$ is extremal for $n=2$.

Consider now the general case where $n\geq3$. If $[F]$ is not extremal in $\overline{\text{Nef}}_1(\Mbar{g}{n})$ then there exists a non-trivial nef decomposition implying the existence of a nef curve with class $[F^{\underline{t}}]$ for $\underline{t}=(t_1,\dots,t_{n-1})\ne \underline{0}$ where
 $$[F^{\underline{t}}]\cdot\psi_i=[F^{\underline{t}}]\cdot\delta_{0:\{i,n\}}=1+t_i,\text{ for $i=1,\dots,n-1$} \hspace{1cm}[F^{\underline{t}}]\cdot \psi_n=2g-(n-3),  $$
 Let $\pi_j:\Mbar{g}{n}\longrightarrow\Mbar{g}{2}$ for $j=1,\dots n-1$ be the morphism forgetting all but the $j$th and $n$th points. As the push forward of any nef curve class is nef and ${\pi_j}_*[F]=[F]$ is extremal in $\overline{\text{Nef}}_1(\Mbar{g}{2})$ and we find 
 $${\pi_j}_*[F^{\underline{t}}]=k_j[F]$$
 for $0\leq k_j\leq 1$ for each $j=1,\dots n-1$. Observe that~\cite{AC}
 $$\pi_j^*\psi_1=\psi_j-\sum_{j\in S,n\nin S}\delta_{0:S},\hspace{0.5cm}\pi_j^*\psi_2=\psi_n-\sum_{n\in S,j\nin S}\delta_{0:S},\hspace{0.5cm}\pi_j^*\delta_{0:\{1,2\}}=\sum_{j,n\in S}\delta_{0:S}$$
 Hence a application of the projection formula for each choice of $j$ this yields
 $$ 2g-1-\sum_{i\ne j}t_i=k_j(2g-1)\hspace{0.5cm}\text{ and }\hspace{0.5cm}1+t_j=k_j.$$
 Giving $n-1$ linearly independent relations in the $t_j$,
 $$(2g-2)t_j+\sum_{i=1}^{n-1}t_i=0$$
 for each $j=1,\dots n-1$. Hence $[F^{\underline{t}}]$ is nef if and only if all $t_i=0$ providing a contradiction and hence $[F]$ is extremal.
\end{proof}

This results in the main result of this paper.

\NonPoly*

\begin{proof}
Theorem~\ref{Thm:Extremal} shows $[F]$ is extremal and Theorem~\ref{Thm:dual} shows the pseudoeffective dual space has corank at least $2$. The cone $\overline{\text{Eff}}^1(\Mbar{g}{n})$ is dual to $\overline{\text{Nef}}_1(\Mbar{g}{n})$ and hence both cones are not rational polyhedral.
\end{proof}



\bibliographystyle{plain}
\bibliography{base}
\end{document}